\documentclass[final,leqno]{siamltex}

\usepackage{amsmath}

\usepackage{amsmath}
\usepackage{amssymb}
\usepackage{graphicx}
\usepackage{tikz}
 \numberwithin{dummy}{section}

\newtheorem{algorithm}{Algorithm}
\newtheorem{remark}{Remark}

\newcommand{\bq}{{\bf q}}

\newcommand{\bx}{{\bf x}}

\def\Q{{\mathbb Q}}

\def\T{{\mathcal  T}}
\def\E{{\mathcal  E}}

\def\pT{{\partial T}}
\def\l{{\langle}}
\def\r{{\rangle}}

\def\T{{\mathcal  T}}
\def\E{{\mathcal  E}}

\def\bn{{\bf n}}
\def\bq{{\bf q}}

\def\3bar{{|\hspace{-.02in}|\hspace{-.02in}|}}

\title{A conforming discontinuous Galerkin finite element method: Part II}
\author{Xiu Ye\thanks{Department of
Mathematics, University of Arkansas at Little Rock, Little Rock, AR
72204 (xxye@ualr.edu). This research was supported in part by
National Science Foundation Grant DMS-1620016.}
\and
Shangyou Zhang\thanks{Department of
Mathematical Sciences, University of Delaware, Newark, DE 19716 (szhang@udel.edu).}
}

\begin{document}

\maketitle

\begin{abstract}
A conforming discontinuous Galerkin (DG) finite element method has been introduced in \cite{cdg1} on simplicial meshes, which has the flexibility of using discontinuous approximation  and the simplicity in formulation of  the classic continuous finite element method. The goal of this paper is to extend the conforming DG finite element method in \cite{cdg1} so that it  can  work on general polytopal meshes by designing weak gradient $\nabla_w$ appropriately. Two different conforming DG formulations on polytopal meshes are introduced which handle boundary conditions differently.
Error estimates of optimal order are established for the corresponding conforming DG approximation in both a discrete $H^1$ norm
and the $L^2$ norm. Numerical results are presented to confirm the theory.
\end{abstract}


\pagestyle{myheadings}

\section{Introduction}
We consider Poisson equation with a homogeneous Dirichlet boundary condition in $d$ dimension
  as our model problem for the sake of clear presentation. This conforming DG method can also be used to  solve other elliptic problems.
The Poisson problem  seeks an unknown function $u$ satisfying
\begin{eqnarray}
-\Delta u&=&f\quad \mbox{in}\;\Omega,\label{pde}\\
u&=&0\quad\mbox{on}\;\partial\Omega,\label{bc}
\end{eqnarray}
where $\Omega$ is a bounded polytopal domain in $\mathbb{R}^d$.

The weak form of the problem (\ref{pde})-(\ref{bc}) is given as follows: find $u\in H_0^1(\Omega)$
such that
\begin{eqnarray}
(\nabla u,\nabla v)=(f,v)\quad \forall v\in
H_0^1(\Omega).\label{weakform}
\end{eqnarray}

The $H^1$ conforming finite element method for the problem (\ref{pde})-(\ref{bc}) keeps the same simple form as in (\ref{weakform}): find $u_h\in V_h\subset H^1_0(\Omega)$
such that
\begin{eqnarray}
(\nabla u_h,\nabla v)=(f,v)\quad \forall v\in V_h,\label{cfe}
\end{eqnarray}
where $V_h$ is a finite dimensional subspace of $H_0^1(\Omega)$.
The functions in $V_h$ are required to be continuous that makes the
classic conforming finite element formulation (\ref{cfe}) less flexible in  element construction and in mesh generation. These limitations are caused by strong continuity requirement of  functions in finite element spaces. A solution to avoid these limitations is  using discontinuous  functions in finite element spaces.

Researchers started to use discontinuous approximation in finite element procedure in the early 1970s \cite{Babu73, bos, DoDu76,ReHi73, Whee78}.
Local discontinuous Galerkin methods were introduced in \cite{cs1998}.   Then a paper \cite{abcm} in 2002 provides a unified analysis of discontinuous Galerkin  finite element methods for Poisson equation.
Since then, many new finite element methods with discontinuous approximations have been developed such as
hybridizable discontinuous Galerkin  method \cite{cgl}, mimetic finite differences method \cite{Lipnikov2011},
hybrid high-order   method \cite{de}, weak Galerkin  method \cite{wy}  and references therein.

One obvious disadvantage of discontinuous finite element methods is their rather complex
formulations which are often necessary to ensure connections of discontinuous solutions across element boundaries.
The purpose of this paper is to obtain a finite element
    formulation close to its original PDE weak form (\ref{weakform})
    for discontinuous polynomials.
We believe that finite element formulations for discontinuous approximations
     can be as simple as follows:
\begin{equation}\label{dfe}
(\nabla_w u_h,\nabla_w v)=(f,v)\quad\forall v\in V_h,
\end{equation}
if  $\nabla_w$, an approximation of gradient,  is appropriately defined for  discontinuous polynomials in $V_h$.
The formulation (\ref{dfe}) can be viewed as a counterpart of (\ref{weakform}) for discontinuous approximations.

In \cite{cdg1}, we have developed a discontinuous finite element method that has an ultra simple weak formulation (\ref{dfe}) on triangular/tetrahedal meshes for any polynomial degree $k\ge 1$.
The formulation (\ref{dfe}) has also been achieved for a WG method defined in \cite{wy} on triangular/tetrahedral meshes.
The lowest order WG method developed in \cite{wy} has been improved in \cite{liu} for convex polygonal meshes,
in which non-polynomial functions are used for computing weak gradient.

The purpose of this paper is to extend the conforming DG in \cite{cdg1} so that it can work on general polytopal meshes. The idea is to raise the degree of polynomials used to compute weak gradient $\nabla_w$.  Using higher degree polynomials in computation of weak gradient will not change
the size, neither the global sparsity of the stiffness matrix. On the other side, the simple formulation of conforming DG (\ref{dfe}) will reduce programming complexity significantly. In this paper, two conforming DG formulations on polytopal mesh are introduced for the equations (\ref{pde})-(\ref{bc}). These two methods are different in handling the homogeneous boundary condition. Optimal order error estimates are established for the corresponding
conforming DG approximations in both a discrete $H^1$ norm and the  $L^2$ norm. Numerical results are presented verifying the theorem.

\section{Finite Element Method}\label{Section:mwg}

In this section, we will introduce the conforming DG method.
For any given polygon $D\subseteq\Omega$, we use the standard
definition of Sobolev spaces $H^s(D)$ with $s\ge 0$. The associated inner product,
norm, and semi-norms in $H^s(D)$ are denoted by
$(\cdot,\cdot)_{s,D}$, $\|\cdot\|_{s,D}$, and $|\cdot|_{s,D}$, respectively. When $s=0$, $H^0(D)$ coincides with the space
of square integrable functions $L^2(D)$. In this case, the subscript
$s$ is suppressed from the notation of norm, semi-norm, and inner
products. Furthermore, the subscript $D$ is also suppressed when
$D=\Omega$.

Let ${\mathcal T}_h$ be a partition of the domain $\Omega$ consisting of
polygons in two dimension or polyhedra in three dimension satisfying
a set of conditions specified in \cite{wymix} and additional conditions specified in Lemma \ref{poly}. Denote by
${\mathcal E}_h$ the set of all edges/faces in ${\mathcal T}_h$, and let
${\mathcal E}_h^0={\mathcal E}_h\backslash\partial\Omega$ be the set of all
interior edges/faces. For simplicity, we will use term edge for edge/face without confusion.

For simplicity, we adopt the following notations,
\begin{eqnarray*}
(v,w)_{\T_h} &=& \sum_{T\in\T_h}(v,w)_T=\sum_{T\in\T_h}\int_T vw d\bx,\\
 \l v,w\r_{\partial\T_h}&=&\sum_{T\in\T_h} \l v,w\r_\pT=\sum_{T\in\T_h} \int_\pT vw ds.
\end{eqnarray*}
Let $P_k(K)$ consist all the polynomials degree less or equal to $k$ defined on $K$.

\begin{algorithm}
A conforming DG finite element method for the problem (\ref{pde})-(\ref{bc})
seeks $u_h\in V_h$  satisfying
\begin{eqnarray}
(\nabla_w u_h,\nabla_w v)_{\T_h} &=&(f,\;v)\quad\forall v\in V_h.\label{mwg}
\end{eqnarray}
\end{algorithm}

The weak gradient $\nabla_w$ in the equation (\ref{mwg}) is defined as follows \cite{mwg, mwg1,wy,wymix}.
For a given $T\in\T_h$ and a function $v\in V_h+H_0^1(\Omega)$, the weak gradient $\nabla_wv\in [P_j(T)]^d$ on $T$ satisfies the following equation,
\begin{equation}\label{d-d}
(\nabla_{w} v, \bq)_T = -(v,\nabla\cdot \bq)_T+ \langle \{v\},
\bq\cdot\bn\rangle_{\partial T}\qquad \forall \bq\in [P_j(T)]^d,
\end{equation}
where $j$ and $\{v\}$ will be defined later.

In the following, we will introduce two finite element formulations by choosing  the vector spaces $V_h$ and the definition of average $\{\cdot\}$ differently.

Let $T_1$ and $T_2$ be two polygons/polyhedrons
sharing $e$ if $e\in\E_h^0$.  For $e\in\E_h$ and $v\in V_h+H_0^1(\Omega) $, the jump $[v]$ is defined as
\begin{equation}\label{jump}
[v]=v\quad {\rm if} \;e\subset \partial\Omega,\quad [v]=v|_{T_1}-v|_{T_2}\;\; {\rm if} \;e\in\E_h^0.
\end{equation}
The order of $T_1$ and $T_2$ is not essential.

{\bf Case 1. Strongly enforce boundary condition}

In this case, $V_h$ is defined  for $k\ge 1$ as
\begin{equation}\label{Vh1}
V_h=\left\{ v\in L^2(\Omega):\ v|_{T}\in
P_{k}(T)\;\; T\in\T_h, \quad v|_{\partial \Omega}=0 \right\}.
\end{equation}

For $e\in\E_h$ and $v\in V_h+H_0^1(\Omega)$, the average $\{v\}$ is defined as
\begin{equation}\label{avg1}
\{v\}=v\quad {\rm if} \;e\subset \partial\Omega,\quad \{v\}=\frac12(v|_{T_1}+v|_{T_2})\;\; {\rm if} \;e\in\E_h^0.
\end{equation}

{\bf Case 2. Weakly enforce boundary condition}

Here, $V_h$ is defined for $k\ge 1$ as
\begin{equation}\label{Vh}
V_h=\left\{ v\in L^2(\Omega):\ v|_{T}\in
P_{k}(T),\;\; T\in\T_h\right\}.
\end{equation}

For $e\in\E_h$ and $v\in V_h+H_0^1(\Omega)$, the average $\{v\}$ is defined  as
\begin{equation}\label{avg2}
\{v\}=0\quad {\rm if} \;e\subset \partial\Omega,\quad \{v\}=\frac12(v|_{T_1}+v|_{T_2})\;\; {\rm if} \;e\in\E_h^0.
\end{equation}

\begin{remark}
For the finite element formulation (\ref{mwg}) associated with Case 1, we assume that each element  $T\in \mathcal {T}_h$ has no more than two edges on $\partial\Omega$  in 2D,  or no more than 3 faces on $\partial\Omega$ in 3D. This requirement is only needed for  error analysis.
In practice, we cannot find any meshes consisting of elements sharing  more than two edges in 2D and three faces in 3D with $\partial\Omega$ after any mesh refinement.
\end{remark}


Let $\Q_h$ be the element-wise defined $L^2$ projection onto $[P_{j}(T)]^d$ on each element $T$.

\begin{lemma}
Let $\phi\in H_0^1(\Omega)$, then on  $T\in\T_h$
\begin{equation}\label{key}
\nabla_w\phi =\Q_h\nabla\phi.
\end{equation}
\end{lemma}
\begin{proof}
Using (\ref{d-d}) and  integration by parts, we have that for
any $\bq\in [P_{j}(T)]^d$
\begin{eqnarray*}
(\nabla_w \phi,\bq)_T &=& -(\phi,\nabla\cdot\bq)_T
+\langle \{\phi\},\bq\cdot\bn\rangle_{\pT}\\
&=& -(\phi,\nabla\cdot\bq)_T
+\langle \phi,\bq\cdot\bn\rangle_{\pT}\\
&=&(\nabla \phi,\bq)_T=(\Q_h\nabla\phi,\bq)_T,
\end{eqnarray*}
which implies the desired identity (\ref{key}).
\end{proof}

\section{Well Posedness}

We start this section by introducing a semi-norms $\3bar v\3bar$ and a norm  $\|v\|_{1,h}$
   for any $v\in V_h+H_0^1(\Omega)$ as follows:
\begin{eqnarray}
\3bar v\3bar^2 &=& \sum_{T\in\T_h}(\nabla_wv,\nabla_wv)_T, \label{norm2}\\
\|v\|_{1,h}^2&=&\sum_{T\in \T_h}\|\nabla v\|_T^2+\sum_{e\in\E_h}h_e^{-1}\|[v]\|_{e}^2.\label{norm3}
\end{eqnarray}

For any function $\varphi\in H^1(T)$, the following trace
inequality holds true (see \cite{wymix} for details):
\begin{equation}\label{trace}
\|\varphi\|_{e}^2 \leq C \left( h_T^{-1} \|\varphi\|_T^2 + h_T
\|\nabla \varphi\|_{T}^2\right).
\end{equation}

\begin{lemma}\label{poly} Let $T$ be a convex
     $(n+1)$-polygon/polyhedron of size $h_T$ with  edges/faces
   $e$, $e_1$, \dots, and $e_{n}$, satisfying minor angle and length conditions
    to be specified in the proof below.
 For a given polynomial $q_0\in P_k(e)$,  we define a polynomial $q\in P_{k+n}(T)$ by
\begin{align} \label{p1} q&=\lambda_1\cdots \lambda_n q_1, \quad \hbox{where $q_1\in P_k(T)$ satisfying}\\
       \label{p2}  \l q-q_0, p\r_e &=0   \quad\forall p\in P_k(e), \\
       \label{p3}  (q,p)_T& =0\quad\forall p\in P_{k-1}(T),
\end{align} where $\lambda_i\in P_1(T)$ vanishes on $e_i$ and assumes value 1 at
   the barycenter of $e$. Then it holds that
\begin{align} \label{p4} \|q\|_T &\le C h_T^{1/2} \|q_0\|_e,
\end{align} where the nonzero constant is defined in \eqref{p-b} below, independent of $T$ and $q_0$.
\end{lemma}

\def\a#1{\begin{align*}#1\end{align*}}\def\an#1{\begin{align}#1\end{align}}

\begin{proof}   First the linear system \eqref{p2}--\eqref{p3} of equation is square,  of size $\dim P_k$.
To show its existence and uniqueness of solution,  we need only to show the uniqueness.
Let $q_0=0$ and $p=q_1$ in \eqref{p2}.  It follows that $q_1\equiv 0$ on $e$ and
    $q_1=\lambda_0q_2$ for some $q_2\in P_{k-1}(T)$ because the weight is
    positive in the weighted $L^2(e)$ inner product.
  Here $\lambda_0\in P_1(T)$,  $\lambda_0|_e=0$,  and $\max_T \lambda_0 = 1$.
Next letting $p=q_2$ in \eqref{p3},  due to a positive weight $\prod_{i=0}^n \lambda_i$ on $T^0$,
   we have $q_2=0$.

If $e_i$ is a neighboring edge/face of $e$,  then
\begin{align*} \lambda_i|_e &=\frac {2  } { h_e} x
\end{align*} where  $h_e$ is the doubled distance from the barycenter of $e$
     to $e_i$ along/on $e$ and
     $x$ is the distance from a point on $e$ to $e_i$ along (2D) or on (3D) $e$.
   For simplicity,  we assume this $h_e$ is also  the size of $e$ (it is indeed in 2D).
To avoid too many constants,  we assume $h_e\ge h_T/4$.
Then
\begin{align}\label{l-i}
    \max_{T} \lambda_i&=\frac{h_{\perp e_i}(T) } { (h_e/2)\sin\alpha_i }
    \le \frac{h_T} { (h_e/2)\sin\alpha_i } \le \frac 8{  \sin\alpha_i }
      \le \frac 8{  \sin\alpha_0 },
\end{align} where $\pi-\alpha_i$ (for some $\alpha_i\ge \alpha_0>0$
   and $\alpha_i\le \pi-\alpha_0$) is the angle
   between $e$ and $e_i$,  $h_{\perp e_i}(T)$ is the maximal distance of points on $T$
   to $e_i$ in the direction orthogonal to $e_i$.
Let $e_1$,\dots, $e_m$ are all the neighboring edges/faces of $e$, $m=2$ in 2D, and $m\le n$.
For a lower bound,  we have
\an{ \label{p-1} \lambda_i|_{T_0}
        & \ge \begin{cases}  \frac {15}{16}  & \hbox{if } \alpha_i \le \pi/2,\\
            1-\frac{\sqrt d}{16 \sin \alpha_i}
          \ge \frac12  & \hbox{if } \alpha_i > \pi/2, \end{cases}
   } where $T_0$ is a square/cube at middle of $e$ with size $h_{e}/16$, cf. Figure \ref{s-ball}.
We note that other than triangles, $\alpha_i\le \pi/2$ for most other polygons.
Here in \eqref{p-1},  we assumed $\sin\alpha_0 \ge \sqrt{d}/8$, where $d$ is the space
  dimension, 2 or 3.

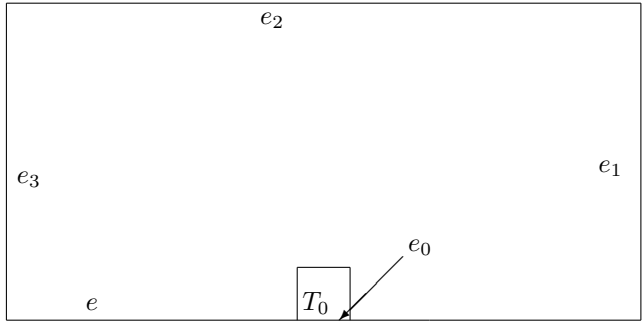
\begin{figure}[h!]
 \begin{center} \setlength\unitlength{2pt}
\begin{picture}(120,60)(0,0)
 \put(0,0){\line(1,0){80}}\put(15,2){$e$}
 \put(55,0){\line(0,1){10}}\put(56,2){$T_0$}
 \put(75,12){\vector(-1,-1){12}} \put(76,13){$e_0$}
 \put(55,10){\line(1,0){10}}
 \put(65,10){\line(0,-1){10}}
 \put(80,0){\line(1,0){40}}
 \put(120,0){\line(0,1){60}}\put(112,28){$e_1$}
 \put(0,60){\line(1,0){120}}\put(48,56){$e_2$}
 \put(0,0){\line(0,1){60}}\put(2,26){$e_3$}
 \end{picture}\end{center}
\caption{Size $|e_0|=|e|/16=e_h/8$, and $T_0$ is square of size $|e_0|$. }\label{s-ball}
\end{figure}

For non-neighboring edges $e_j$,  we have
\a{ \lambda_j|_{e_1} = \begin{cases} 1 & \hbox{if } e_j \parallel e_1,\\
           \frac {2(x + x_j)}{h_{e_1}+x_j}  & \hbox{otherwise, }\end{cases}
     }
where $x$ is the arc-length parametrization on $e$ toward the extended intersection of
  $e$ and $e_i$,  $x_j$ is the distance on $e$ from the an boundary point of $e$ to
   the intersection.
Supposing $e_i$ is the only edge/polygonal between $e$
  and $e_j$,  $x_j=h_{e_i}(\cos \alpha_i-\cos(\alpha_i+\alpha_j))$.
Because $x_j\ge 0$,  it follows that
\an{\label{l21} \max_T \lambda_j = \frac{h_{\perp e_j}(T) }{ (h_{e}/2)\sin \alpha_i}
      \le \frac{2 h_T }{ ( h_{e}+x_j) \sin (\alpha_i+\alpha_j)}
      \le \frac 8 {   \sin \alpha_0}. }
For a lower bound, because $x_j>0$ and $e_i$ is an edge/polygon in between,  we have
\an{\label{l22}  \lambda_j|_{T_0}
        & \ge  \lambda_i|_{T_0} \ge \frac 12.  }
Together,  we have, noting $\lambda_0|_T\le 1$,
\an{\label{bound} \lambda_1 \cdots\lambda_n |_{T_0} \ge \frac{1} {2^{ n }},
     \quad\hbox{ and } \
   \lambda_0 \lambda_1\cdots\lambda_n |_T \le \frac{8^{ n }} {\sin^{ n } \alpha_0}.  }

Let $\tilde q_1\in P_k(e)$ be the solution in \eqref{p2}.
Letting $\tilde p=q_1$ in \eqref{p2}, by \eqref{bound},  we get
\a{ \frac 1{16^{2k}} \frac 1{2^{n}}  \|\tilde  q_1\|_e^2
            &\le  \frac 1{2^{n}}  \| \tilde q_1\|_{e_0} ^2
             \le
              \l \lambda_1\cdots \lambda_n \tilde q_1,\tilde q_1 \r_e \\
                 & =  \l  q_0,\tilde q_1 \r_e \le \| q_0\|_0 \| \tilde q_1\|_0,
   }
where in the first step we use the fact $ q_1$ is a degree $k$ polynomial.
We view $\tilde q_1\in P_k(e)$  as defined on the whole line/plane passing through $e$.
We extend this polynomial to a polynomial $\tilde q_1$ in $P_k(\mathbb{R}^d)$,
    by letting it be constant in the direction orthogonal to $e$.
In particular, we have, as $T\subset S_T$ and $e\subset S_e$,
\a{ \|\tilde q_1\|_T^2 & \le \|\tilde q_1\|_{S_T}^2= h_T \|\tilde q_1\|_{S_e}^2
    \le \left(\frac {h_T}{h_e}\right)^{2k}
     h_T \|\tilde q_1\|_{e}^2 \\
    & \le 4^{2k} h_T \|\tilde q_1\|_{e}^2
  \le  2^{4k} h_T ( 2^{8k+n} \|q_0\|_{e}) ^2 ,
 } where $S_T$ is a square/cube of size $h_T$ containing $T$, with one side $S_e$ which
   contains $e$.

Rewriting \eqref{p1} in terms of this extended $\tilde q_1$, we have
\begin{align*} q=\lambda_1\cdots \lambda_n (\lambda_0 q_2+   \tilde q_1)
\end{align*} for some $q_2\in P_{k-1}(T)$.
Letting $p=q_2$ in \eqref{p3},   by \eqref{bound},  we have
\begin{align*}
  \|q_2\|_T^2 &\le \left( {h_T}/{h_{e_0}}\right)^{2k-2}  \|q_2\|_{T_0}^2
     \le 64^{2k-2} \frac{8^{ n }} {\sin^{ n } \alpha_0}  (\lambda_1 \cdots \lambda_n  q_2,  q_2)_{T_0}
     \\&\le  \frac{2^{ 3n+12k-12 }} {\sin^{ n } \alpha_0}
          \frac {2 h_T} {h_{e_0}}   (\lambda_1 \cdots \lambda_n \lambda_0 q_2,  q_2)_{T_{0,0}}
     \\&\le  \frac{2^{ 3n+12k-5 }} {\sin^{ n } \alpha_0}
              (\lambda_1 \cdots \lambda_n \lambda_0 q_2,  q_2)_{T }\\
      & =\frac{2^{ 3n+12k-5 }} {\sin^{ n } \alpha_0}  (\lambda_1 \cdots \lambda_n \tilde q_1,-q_2)_{T}
   \\ & \le \frac{2^{ 3n+12k-5 }} {\sin^{ n } \alpha_0} 2^n \|\tilde q_1\|_T \|q_2\|_T,
\end{align*} where $T_{0,0}$ is the top half of $T_0$,  cf. Figure \ref{s-ball}.
Then,
\begin{align*}
  \|q\|_T^2 & =(\lambda_1^2\cdots \lambda_n^2 (\lambda_0 q_2-\tilde q_1), (\lambda_0 q_2-\tilde q_1))_T
     \\& \le \frac{8^{ 2n }} {\sin^{ 2n } \alpha_0}
      ( (\lambda_0 q_2-\tilde q_1), (\lambda_0 q_2-\tilde q_1))_T
     \\& \le \frac{8^{ 2n }} {\sin^{ 2n } \alpha_0} 2  (\| \lambda_0 q_2\|_T^2 + \|\tilde q_1\|_T^2 )
   \\ & \le  \frac{2^{ 6n+1 }} {\sin^{ 2n } \alpha_0} (\| q_2\|_T^2 + \|\tilde q_1\|_T^2 ) ,
\end{align*} where  $\lambda_0\le 1$ on $T$.
Finally, combining above three bounds,  we get
\an{  \label{p-b} \begin{aligned}
  \|q\|_T &  \le \frac{2^{ 3n+1/2 }} {\sin^{ n } \alpha_0}
            \left( (\frac{2^{4n+12k-5 }} {\sin^{ n } \alpha_0} )^2+1 \right)^{\frac 12}
               \|\tilde  q_1 \|_T \\
         &\le  \frac{2^{10k+ 4n+1/2 }} {\sin^{ n } \alpha_0}
            \left( (\frac{2^{4n+12k-5 }} {\sin^{ n } \alpha_0} )^2+1 \right)^{\frac 12}
  h_T^{ 1/2} \|q_0\|_{e} \\
   &=: C h_T^{ 1/2} \|q_0\|_{e}.
\end{aligned}  }
   The proof is completed.
\end{proof}

\begin{lemma} There exist two positive constants $C_1$ and $C_2$ independent of mesh size $h$ such
that for any $v\in V_h$, we have
\begin{equation}\label{happy}
C_1 \|v\|_{1,h}\le \3bar v\3bar \leq C_2 \|v\|_{1,h}.
\end{equation}
\end{lemma}

\medskip

\begin{proof}
For any $v\in V_h$, it follows from the definition of
weak gradient (\ref{d-d}) and integration by parts that for all $\bq\in [P_{j}(T)]^d$
\begin{eqnarray}
(\nabla_wv,\bq)_T&=&-(v, \nabla\cdot\bq)_T+\l \{v\},
\bq\cdot\bn\r_\pT\nonumber\\
&=&(\nabla v,\bq)_T-\l v-\{v\}, \bq\cdot\bn\r_\pT. \label{n-1}
\end{eqnarray}
By letting $\bq=\nabla_w v$ in (\ref{n-1}) we arrive at
\begin{eqnarray*}
(\nabla_wv,\nabla_w v)_T=(\nabla v,\nabla_w v)_T-\l v-\{v\},
\nabla_w v\cdot\bn\r_\pT.
\end{eqnarray*}
It is easy to see that the following equations hold true for $\{v\}$ defined in both (\ref{avg1}) and (\ref{avg2}) on $T$ with $e\subset\pT$,
\begin{equation}\label{jp}
\|v-\{v\}\|_e=\|[v]\|_e\quad {\rm if} \;e\subset \partial\Omega,\quad\|v-\{v\}\|_e=\frac12\|[v]\|_e\;\; {\rm if} \;e\in\E_h^0.
\end{equation}
From (\ref{jp}), (\ref{trace}) and the inverse inequality
we have
\begin{eqnarray*}
\|\nabla_wv\|^2_T &\le& \|\nabla v\|_T \|\nabla_w v\|_T+\|
v-\{v\}\|_\pT \|\nabla_w v\|_\pT\\
&\le& \|\nabla v\|_T \|\nabla_w v\|_T+ Ch_T^{-1/2}\|
v-\{v\}\|_\pT \|\nabla_w v\|_T\\
&\le& \|\nabla v\|_T \|\nabla_w v\|_T+ Ch_T^{-1/2}\|[v]\|_{\pT} \|\nabla_w v\|_T
\end{eqnarray*}
which implies
$$
\|\nabla_w v\|_T \le C \left(\|\nabla v\|_T +Ch_T^{-1/2}\|[v]\|_{\pT}\right),
$$
and consequently
$$\3bar v\3bar \leq C_2 \|v\|_{1,h}.$$

Next we will prove $C_1 \|v\|_{1,h}\le \3bar v\3bar $.
For $v\in V_h$ and $\bq\in [P_j(T)]^d$, by (\ref{d-d}) and integration by parts, we have
\begin{equation}\label{n2}
(\nabla_w v,\bq)_T=(\nabla v,\bq)_T+\l \{v\}-v, \bq\cdot\bn\r_\pT.
\end{equation}
We like to find  $\bq_0\in [P_j(T)]^d$ such that,
\begin{equation}\label{2e}
(\nabla v,\bq_0)_T=0,\quad \l \{v\}-v,\bq_0\cdot\bn\r_{\pT\setminus e}=0,
   \ \hbox{ and  } \ \l \{v\}-v,\bq_0\cdot\bn\r_e = \|\{v\}-v\|_e^2,
\end{equation}
and
\begin{align}\label{i-b}
\| \bq_0\|_{ T} \le C h_T^{1/2} \| \{v\}-v  \|_{ {e}}.
\end{align}
Letting $q_0= \{v\}-v$ in \eqref{p2}, there exists  a $q\in P_{n+k-1}(T)$ (i.e. $j=n+k-1$)
  such that \eqref{p2}--\eqref{p4} hold, where $n$ is the number of the edges/faces on a polygon/polyhadron.
Without loss of generality,  let $\bn=\l n_1, \cdots, n_d\r$ for some $n_1\ne 0$.
We then let $\bq_0=\l q/n_1, 0, \cdots, 0 \r$,  which satisfies \eqref{2e} and  \eqref{i-b} by Lemma \ref{poly}.
Substituting  $\bq_0$ into (\ref{n2}), we get
\begin{equation}\label{n3}
(\nabla_wv,\bq_0)_T=\|\{v\}-v\|_e^2.
\end{equation}
It follows from Cauchy-Schwarz inequality that
\[
\|\{v\}-v\|_e^2\le C\|\nabla_w v\|_T\|\bq_0\|_T\le Ch_T^{1/2}\|\nabla_w v\|_T\|\{v\}-v \|_e,
\]
which gives
\begin{equation}\label{n4}
h_T^{-1/2}\|\{v\}-v\|_\pT\le C\|\nabla_w v\|_T.
\end{equation}
Using (\ref{jp}) and summing the both sides of (\ref{n4}) over $T$, we obtain
\begin{equation}\label{nn4}
\sum_{e\in\E_h}h_e^{-1}\|[v]\|_e^2\le C\3bar v\3bar^2.
\end{equation}
It follows from the trace inequality, the inverse inequality and (\ref{n4}),
$$
\|\nabla v\|_T^2 \leq \|\nabla_w v\|_T \|\nabla v\|_T
+Ch_T^{-1/2}\| \{v\}-v\|_\pT \|\nabla v\|_T\le C\|\nabla_w v\|_T \|\nabla v\|_T,
$$
which implies
\begin{equation}\label{nn5}
\sum_{T\in\T_h}\|\nabla v\|_T^2\le C\3bar v\3bar^2.
\end{equation}

Combining (\ref{nn4}) and (\ref{nn5}),
 we prove  the lower bound of (\ref{happy}) and complete the proof of the lemma.
\end{proof}

\section{Error Estimates in Energy Norm}

We start this section by defining some approximation operators. 
We will call any element $T\in\T_h$, that has one or two edges on $\partial\Omega$,  boundary element in 2D. Then we will define $I_hu$, an interpolation of $u$, on boundary elements.  $I_hu$ for 3D can be constructed in a similar fashion. For a boundary element $T$, let $T_0\subset T$ be a triangle such that $\pT\cap\partial\Omega = \pT_0\cap\partial\Omega$. Let $I_hu$ be $k$th order interpolation of $u$ on $T_0$.

\begin{lemma}
For any  boundary element $T\in\T_h$,  one has
\begin{equation}\label{approx}
\|u-I_hu\|_T +h_T\|\nabla (u-I_hu)\|_T\le Ch^{k+1}|u|_{k+1,T}.
\end{equation}
\end{lemma}

\begin{proof}
For any boundary element $T\in\T_h$, by the construction of $I_hu$, one has
\begin{equation}\label{a1}
\|u-I_hu\|_{T_0} +h_T\|\nabla (u-I_hu)\|_{T_0}\le Ch^{k+1}|u|_{k+1,T_0}.
\end{equation}
Let $Q_0$ be the $L^2$ projection onto $P_k(T)$. The following estimate holds \cite{mwy-se}
\begin{equation}\label{a2}
\|u-Q_0u\|_T +h_T\|\nabla (u-Q_0u)\|_T\le Ch^{k+1}|u|_{k+1,T}.
\end{equation}
By the triangle inequality, then
\begin{equation}\label{a3}
\|u-I_hu\|_T\le  \|u-Q_0u\|_T+\|Q_0u-I_hu\|_T.
\end{equation}
By the domain inverse inequality \cite{mwwy, mwy-biharm}
   and under necessary regularity assumption of the mesh $\T_h$, we have
\begin{equation}\label{a4}
\|Q_0u-I_hu\|_T \le C\|Q_0u-I_hu\|_{T_0}\le C (\|Q_0u-u\|_{T_0}+\|u-I_hu\|_{T_0}).
\end{equation}
Combining (\ref{a1})-(\ref{a4}) yields
\[
\|u-I_hu\|_{T}\le Ch^{k+1}|u|_{k+1,T}.
\]
Similarly, we can prove the second part of the estimate in (\ref{approx}) and finish the proof of the lemma.
\end{proof}

Now we define $Q_hu\in V_h$, an approximation of $u$ for the two finite element methods associated with Case 1 and Case 2. For the method associated with Case 1, let $Q_hu=Q_0u$ for any $T$ which is not boundary element and $Q_hu=I_hu$ for the boundary element $T$.
For the case 2, define $Q_hu=Q_0u$ for all $T\in\T_h$.

Let $e_h=u-u_h$ and $\epsilon_h=Q_hu-u_h\in V_h$. Next we derive an error equation that $e_h$ satisfies.

\begin{lemma}
For any $v\in V_h$, one has,
\begin{eqnarray}
(\nabla_we_h,\nabla_wv)_{\T_h}=\ell(u,v),\label{ee}
\end{eqnarray}
where
\begin{eqnarray*}
\ell(u,v)&=& \langle (\nabla u-\Q_h\nabla u)\cdot\bn,v-\{v\}\rangle_{\pT_h}.
\end{eqnarray*}
\end{lemma}

\begin{proof}
Testing (\ref{pde}) by  any $v\in V_h$  and using integration by parts and the fact that
$\sum_{T\in\T_h}\langle \nabla u\cdot\bn, \{v\}\rangle_\pT=0$ for $\{v\}$ defined in both (\ref{avg1}) and (\ref{avg2}),  we arrive at
\begin{equation}\label{m1}
(\nabla u,\nabla v)_{\T_h}- \langle
\nabla u\cdot\bn,v-\{v\}\rangle_{\pT_h}=(f,v).
\end{equation}

It follows from integration by parts, (\ref{d-d}) and (\ref{key})  that
\begin{eqnarray}
(\nabla u,\nabla v)_{\T_h}&=&(\Q_h\nabla  u,\nabla v)_{\T_h}\nonumber\\
&=&-(v,\nabla\cdot (\Q_h\nabla u))_{\T_h}+\langle v, \Q_h\nabla u\cdot\bn\rangle_{\partial\T_h}\nonumber\\
&=&(\Q_h\nabla u, \nabla_w v)_{\T_h}+\langle v-\{v\},\Q_h\nabla u\cdot\bn\rangle_{\partial\T_h}\nonumber\\
&=&( \nabla_w u, \nabla_w v)_{\T_h}+\langle v-\{v\},\Q_h\nabla u\cdot\bn\rangle_{\partial\T_h}.\label{j1}
\end{eqnarray}
Combining (\ref{m1}) and (\ref{j1}) gives
\begin{eqnarray}
(\nabla_w u,\nabla_w v)_{\T_h}&=&(f,v)+\ell(u,v).\label{j2}
\end{eqnarray}
The error equation follows from subtracting (\ref{mwg}) from (\ref{j2}),
\begin{eqnarray*}
(\nabla_we_h,\nabla_wv)_{\T_h}=\ell(u,v)\quad \forall v\in V_h.
\end{eqnarray*}
This completes the proof of the lemma.
\end{proof}

\begin{lemma} For any $w\in H^{k+1}(\Omega)$ and $v\in V_h$, we have
\begin{eqnarray}
|\ell(w, v)|&\le&
Ch^{k}|w|_{k+1}\3bar v\3bar.\label{mmm1}
\end{eqnarray}
\end{lemma}

\medskip

\begin{proof}
Using the Cauchy-Schwarz inequality, the trace inequality (\ref{trace}), (\ref{jp}) and  (\ref{happy}), we have
\begin{eqnarray*}
|\ell(w,v)|&=&\left|\sum_{T\in\T_h}\langle (\nabla w-\Q_h\nabla
w)\cdot\bn, v-\{v\}\rangle_\pT\right|\\
&\le & C \sum_{T\in\T_h}\|\nabla w-\Q_h\nabla w\|_{\pT}
\|v-\{v\}\|_\pT\nonumber\\
&\le & C \left(\sum_{T\in\T_h}h_T\|(\nabla w-\Q_h\nabla w)\|_{\pT}^2\right)^{\frac12}
\left(\sum_{e\in\E_h}h_e^{-1}\|[v]\|_e^2\right)^{\frac12}\\
&\le & Ch^{k}|w|_{k+1}\3bar v\3bar,
\end{eqnarray*}
which proves the lemma.
\end{proof}

\smallskip

\begin{lemma}
Let $u\in H^{k+1}(\Omega)$, then
\begin{equation}\label{eee2}
\3baru-Q_hu\3bar\le Ch^k|u|_{k+1}.
\end{equation}
\end{lemma}
\begin{proof}
It follows from (\ref{d-d}), integration  by parts, (\ref{trace}) and (\ref{jp}),
\begin{eqnarray*}
|(\nabla_w(u-Q_hu), \bq)_{T}|&=&|-(u-Q_hu, \nabla\cdot\bq)_{T}+\l u-\{Q_hu\}, \bq\cdot\bn\r_{\pT}|\\
&=&|(\nabla (u-Q_hu), \bq)_{T}+\l Q_hu-\{Q_hu\}, \bq\cdot\bn\r_{\pT}|\\
&\le& \|\nabla (u-Q_hu)\|_T\|\bq\|_T+Ch^{-1/2}\|[Q_hu]\|_\pT\|\bq\|_T\\
&\le& \|\nabla (u-Q_hu)\|_T\|\bq\|_T+Ch^{-1/2}\|[u-Q_hu]\|_\pT\|\bq\|_T\\
&\le& Ch^k|u|_{k+1, T}\|\bq\|_T.
\end{eqnarray*}
Letting $\bq=\nabla_w(u-Q_hu)$ in the above equation and taking summation over $T$, we have
\[
\3baru-Q_hu\3bar\le Ch^k|u|_{k+1}.
\]
We have proved the lemma.
\end{proof}

\begin{theorem} Let $u_h\in V_h$ be the finite element solution of (\ref{mwg}). Assume the exact solution $u\in H^{k+1}(\Omega)$. Then,
there exists a constant $C$ such that
\begin{equation}\label{err1}
\3bar u-u_h\3bar \le Ch^{k}|u|_{k+1}.
\end{equation}
\end{theorem}
\begin{proof}
It is straightforward to obtain
\begin{eqnarray}
\3bar e_h\3bar^2&=&(\nabla_we_h, \nabla_we_h)_{\T_h}\label{eee1}\\
&=&(\nabla_wu-\nabla_wu_h,\nabla_we_h)_{\T_h}\nonumber\\
&=&(\nabla_wQ_hu-\nabla_wu_h,\nabla_we_h)_{\T_h}+(\nabla_wu-\nabla_wQ_hu,\nabla_we_h)_{\T_h}\nonumber\\
&=&(\nabla_we_h,\nabla_w\epsilon_h)_{\T_h}+(\nabla_w(u-Q_hu),\nabla_we_h)_{\T_h}.\nonumber
\end{eqnarray}
We will bound each terms in (\ref{eee1}).
Letting $v=\epsilon_h\in V_h$ in (\ref{ee})  and using (\ref{mmm1}) and (\ref{eee2}), we have
\begin{eqnarray}
|(\nabla_we_h,\nabla_w\epsilon_h)_{\T_h}|&=&|\ell(u,\epsilon_h)|\nonumber\\
&\le& Ch^{k}|u|_{k+1}\3bar \epsilon_h\3bar\nonumber\\
&\le& Ch^{k}|u|_{k+1}\3bar Q_hu-u_h\3bar\nonumber\\
&\le& Ch^{k}|u|_{k+1}(\3bar Q_hu-u\3bar+\3bar u-u_h\3bar)\nonumber\\
&\le& Ch^{2k}|u|^2_{k+1}+\frac14 \3bare_h\3bar^2.\label{eee3}
\end{eqnarray}
The estimate (\ref{eee2}) implies
\begin{eqnarray}
|(\nabla_w(u-Q_hu),\nabla_we_h)_{\T_h}|&\le& C\3bar u-Q_hu\3bar \3bar e_h\3bar\nonumber\\
&\le& Ch^{2k}|u|^2_{k+1}+\frac14\3bar e_h\3bar^2.\label{eee4}
\end{eqnarray}
Combining the estimates (\ref{eee3}) and  (\ref{eee4}) with (\ref{eee1}), we arrive
\[
\3bar e_h\3bar \le Ch^{k}|u|_{k+1},
\]
which completes the proof.
\end{proof}

\section{Error Estimates in $L^2$ Norm}

The standard duality argument is used to obtain $L^2$ error estimate.
Recall $e_h=u-u_h$ and $\epsilon_h=Q_hu-u_h$.
The considered dual problem seeks $\Phi\in H_0^1(\Omega)$ satisfying
\begin{eqnarray}
-\Delta\Phi&=& e_h\quad
\mbox{in}\;\Omega.\label{dual}
\end{eqnarray}
Assume that the following $H^{2}$-regularity holds
\begin{equation}\label{reg}
\|\Phi\|_2\le C\|e_h\|.
\end{equation}

\begin{theorem} Let $u_h\in V_h$ be the finite element solution of (\ref{mwg}).
Assume that the
exact solution $u\in H^{k+1}(\Omega)$ and (\ref{reg}) holds true.
 Then, there exists a constant $C$ such that
\begin{equation}\label{err2}
\|u-u_h\| \le Ch^{k+1}|u|_{k+1}.
\end{equation}
\end{theorem}
\begin{proof}
Testing (\ref{dual}) by $e_h$ and using the fact that $\sum_{T\in\T_h}\langle \nabla
\Phi\cdot\bn, \{e_h\}\rangle_\pT=0$ and (\ref{d-d}) give
\begin{eqnarray*}
\|e_h\|^2&=&-(\Delta\Phi,e_h)\\
&=&(\nabla \Phi,\ \nabla e_h)_{\T_h}-\l
\nabla\Phi\cdot\bn,\ e_h- \{e_h\}\r_{\pT_h}\\
&=&(\Q_h\nabla \Phi,\ \nabla e_h)_{\T_h}+(\nabla\Phi-\Q_h\nabla \Phi,\ \nabla e_h)_{\T_h}-\l
\nabla\Phi\cdot\bn,\ e_h- \{e_h\}\r_{\pT_h}\\
&=&-(\nabla\cdot\Q_h\nabla \Phi,\ e_h)_{\T_h}+\l \Q_h\nabla\Phi\cdot\bn,\ e_h\r_{\pT_h}\\
&+&(\nabla\Phi-\Q_h\nabla \Phi,\ \nabla e_h)_{\T_h}-\l\nabla\Phi\cdot\bn,\ e_h- \{e_h\}\r_{\pT_h}\\
&=&(\Q_h\nabla \Phi,\ \nabla_we_h)_{\T_h}+\l\Q_h\nabla\Phi\cdot\bn,\ e_h-\{e_h\}\r_{\pT_h}\\
&+&(\nabla\Phi-\Q_h\nabla \Phi,\ \nabla e_h)_{\T_h}-\l\nabla\Phi\cdot\bn,\ e_h- \{e_h\}\r_{\pT_h}\\
&=&(\Q_h\nabla \Phi,\ \nabla_we_h)_{\T_h}+(\nabla\Phi-\Q_h\nabla \Phi,\ \nabla e_h)_{\T_h}-\ell(\Phi,e_h).
\end{eqnarray*}
It follows from (\ref{key}) and (\ref{ee})
\begin{eqnarray*}
(\Q_h\nabla \Phi,\ \nabla_we_h)_{\T_h}&=&(\nabla_w \Phi,\;\nabla_w e_h)_{\T_h}\\
&=&(\nabla_w Q_h\Phi,\;\nabla_w e_h)_{\T_h}+(\nabla_w (\Phi-Q_h\Phi),\;\nabla_w e_h)_{\T_h}\\
&=&\ell(u,Q_h\Phi)+(\nabla_w (\Phi-Q_h\Phi),\;\nabla_w e_h)_{\T_h}.
\end{eqnarray*}
Combining the two equations above gives
\begin{eqnarray}
\|e_h\|^2&=&\ell(u,Q_h\Phi)+(\nabla_w (\Phi-Q_h\Phi),\;\nabla_w e_h)_{\T_h}\nonumber\\
&+&(\nabla\Phi-\Q_h\nabla \Phi,\ \nabla e_h)_{\T_h}+\ell(\Phi, e_h).\label{m2}
\end{eqnarray}

Next we will estimate all the terms on the right hand side of (\ref{m2}). Using the Cauchy-Schwarz inequality, the trace inequality (\ref{trace}) and the definitions of $Q_h$ and $\Q_h$
we obtain
\begin{eqnarray*}
|\ell(u,Q_h\Phi)|&\le&\left| \langle (\nabla u-\Q_h\nabla
u)\cdot\bn,\;
Q_h\Phi-\{Q_h\Phi\}\rangle_{\pT_h} \right|\\
&\le& \left(\sum_{T\in\T_h}\|(\nabla u-\Q_h\nabla
u)\|^2_\pT\right)^{1/2}
\left(\sum_{T\in\T_h}\|Q_h\Phi-\{Q_h\Phi\}\|^2_\pT\right)^{1/2}\nonumber \\
&\le& C\left(\sum_{T\in\T_h}h\|(\nabla u-\Q_h\nabla
u)\|^2_\pT\right)^{1/2}
\left(\sum_{T\in\T_h}h^{-1}\|[Q_h\Phi-\Phi]\|^2_\pT\right)^{1/2} \nonumber\\
&\le&  Ch^{k+1}|u|_{k+1}|\Phi|_2.\nonumber
\end{eqnarray*}
It follows from (\ref{err1}) and (\ref{eee2}) that
\begin{eqnarray*}
|(\nabla_w e_h,\ \nabla_w(\Phi-Q_h\Phi))_{\T_h}|&\le& C\3bar e_h\3bar \3bar \Phi-Q_h\Phi\3bar\\
&\le& Ch^{k+1}|u|_{k+1}|\Phi|_2.
\end{eqnarray*}
The norm equivalence (\ref{happy}) implies
\begin{eqnarray*}
|(\nabla\Phi-\Q_h\nabla \Phi,\ \nabla e_h)_{\T_h}|&\le& C(\sum_{T\in\T_h}\|\nabla e_h\|_T^2)^{1/2} (\sum_{T\in\T_h}\|\nabla\Phi-\Q_h\nabla \Phi\|_T^2)^{1/2}\\
&\le& C(\sum_{T\in\T_h}(\|\nabla(u-Q_hu)\|_T^2+\|\nabla(Q_hu-u_h)\|_T^2))^{1/2}\\
&\times &(\sum_{T\in\T_h}\|\nabla\Phi-\Q_h\nabla \Phi\|_T^2)^{1/2}\\
&\le& Ch|\Phi|_2(h^k|u|_{k+1}+\3bar Q_hu-u_h\3bar)\\
&\le& Ch|\Phi|_2(h^k|u|_{k+1}+\3bar u-u_h\3bar+\3bar Q_hu-u\3bar)\\
&\le& Ch^{k+1}|u|_{k+1}|\Phi|_2.
\end{eqnarray*}
Using (\ref{happy}),  (\ref{jp}), (\ref{err1}), and (\ref{eee2}), we obtain
\begin{eqnarray*}
|\ell(\Phi,e_h)|&=&\left|\sum_{T\in\T_h}\l (\Q_h\nabla \Phi-\nabla \Phi)\cdot\bn,\
e_h-\{e_h\}\r_\pT\right|\\
&\le&\sum_{T\in\T_h} h_T^{1/2}\|\Q_h\nabla \Phi-\nabla \Phi\|_\pT  h_T^{-1/2}\|[e_h]\|_\pT\\
&\le&Ch\|\Phi\|_2(\sum_{T\in\T_h}h_T^{-1}(\|[\varepsilon_h]\|^2_\pT+\|[u-Q_hu]\|^2_\pT)^{1/2}\\
&\le&Ch\|\Phi\|_2(\3bar\varepsilon_h\3bar+(\sum_{T\in\T_h}h_T^{-1}\|[u-Q_hu]\|^2_\pT)^{1/2}\\
&\le&Ch\|\Phi\|_2(\3bar e_h\3bar+\3bar u-Q_hu\3bar+Ch^k|u|_{k+1})\\
&\le&  Ch^{k+1 }|u|_{k+1}\|\Phi\|_2.
\end{eqnarray*}
Combining all the estimates above
with (\ref{m2}) yields
$$
\|e_h\|^2 \leq C h^{k+1}|u|_{k+1} \|\Phi\|_2.
$$
The estimate (\ref{err2}) follows from the above inequality and
the regularity assumption (\ref{reg}). We have completed the proof.
\end{proof}

\section{Numerical Example}

 We solve the following Poisson equation on the unit square:
\begin{align} \label{s1} -\Delta u = 2\pi^2 \sin\pi x\sin \pi y,  \quad (x,y)\in\Omega=(0,1)^2,
\end{align} with the boundary condition $u=0$ on $\partial \Omega$.

\begin{figure}[h!]
 \begin{center} \setlength\unitlength{1.25pt}
\begin{picture}(260,80)(0,0)
  \def\tr{\begin{picture}(20,20)(0,0)\put(0,0){\line(1,0){20}}\put(0,20){\line(1,0){20}}
          \put(0,0){\line(0,1){20}} \put(20,0){\line(0,1){20}}  \put(20,0){\line(-1,1){20}}\end{picture}}
 {\setlength\unitlength{5pt}
 \multiput(0,0)(20,0){1}{\multiput(0,0)(0,20){1}{\tr}}}

  {\setlength\unitlength{2.5pt}
 \multiput(45,0)(20,0){2}{\multiput(0,0)(0,20){2}{\tr}}}

  \multiput(180,0)(20,0){4}{\multiput(0,0)(0,20){4}{\tr}}

 \end{picture}\end{center}
\caption{\label{grid1} The first three levels of grids used in the computation of Table \ref{t1}. }
\end{figure}
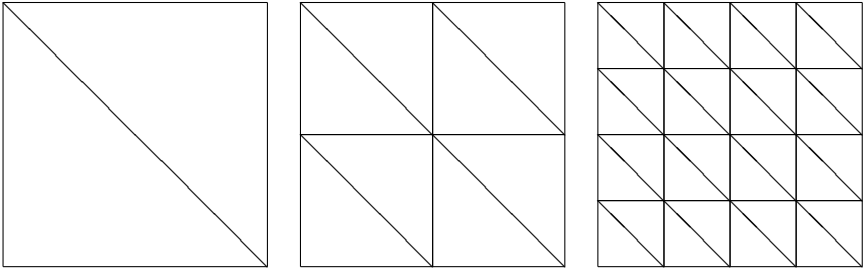

In the first computation, the level one grid consists of two unit right triangles
     cutting from the unit square by a forward
  slash.   The high level grids are the half-size refinements of the previous grid.
The first three levels of grids are plotted in Figure \ref{grid1}.
The error and the order of convergence for the both methods are shown in Tables \ref{t1} and \ref{t2}.
Here on triangular grids,  we let $j=k+1$ defined in (\ref{d-d}) for computing the  weak gradient $\nabla_w v$.
The numerical results confirm the convergence theory.

\begin{table}[h!]
  \centering \renewcommand{\arraystretch}{1.1}
  \caption{Error profiles and convergence rates for \eqref{s1} on triangular grids (Figure \ref{grid1}) }\label{t1}
\begin{tabular}{c|cc|cc|r}
\hline
level & $\|u_h- Q_0 u\|_0 $  &rate & $\3bar u_h- u\3bar $ &rate & dim  \\
\hline
 &\multicolumn{5}{l}{by $P_1$ elements with strongly enforced boundary condition} \\ \hline
 6&   0.5655E-03 & 2.00&   0.8945E-01 & 1.00&    5890\\
 7&   0.1412E-03 & 2.00&   0.4463E-01 & 1.00&   24066\\
 8&   0.3526E-04 & 2.00&   0.2229E-01 & 1.00&   97282\\
 \hline
 &\multicolumn{5}{l}{by $P_1$ elements with weakly enforced boundary condition} \\ \hline
 6&   0.5970E-03 & 2.09&   0.8575E-01 & 0.94&    6144\\
 7&   0.1449E-03 & 2.04&   0.4371E-01 & 0.97&   24576\\
 8&   0.3570E-04 & 2.02&   0.2206E-01 & 0.99&   98304\\
 \hline
 &\multicolumn{5}{l}{by $P_2$ elements with strongly enforced boundary condition} \\ \hline
 6&   0.6635E-05 & 2.99&   0.1797E-02 & 2.00&   11906\\
 7&   0.8314E-06 & 3.00&   0.4489E-03 & 2.00&   48386\\
 8&   0.1040E-06 & 3.00&   0.1122E-03 & 2.00&  195074\\
 \hline
 &\multicolumn{5}{l}{by $P_2$ elements with weakly enforced boundary condition} \\ \hline
 6&   0.6446E-05 & 2.94&   0.1744E-02 & 1.95&   12288\\
 7&   0.8197E-06 & 2.98&   0.4424E-03 & 1.98&   49152\\
 8&   0.1033E-06 & 2.99&   0.1113E-03 & 1.99&  196608\\
 \hline
 &\multicolumn{5}{l}{by $P_3$ elements with strongly enforced boundary condition} \\ \hline
 6&   0.4263E-07 & 4.00&   0.2253E-04 & 3.01&   19970\\
 7&   0.2664E-08 & 4.00&   0.2810E-05 & 3.00&   80898\\
 8&   0.1666E-09 & 4.00&   0.3509E-06 & 3.00&  325634\\
 \hline
 &\multicolumn{5}{l}{by $P_3$ elements with weakly enforced boundary condition} \\ \hline
 6&   0.4311E-07 & 4.02&   0.2193E-04 & 2.97&   20480\\
 7&   0.2679E-08 & 4.01&   0.2772E-05 & 2.98&   81920\\
 8&   0.1670E-09 & 4.00&   0.3485E-06 & 2.99&  327680\\
 \hline
\end{tabular}%
\end{table}%

\begin{table}[h!]
  \centering \renewcommand{\arraystretch}{1.1}
  \caption{Error profiles and convergence rates for \eqref{s1} on triangular grids (Figure \ref{grid1}) }\label{t2}
\begin{tabular}{c|cc|cc|r}
\hline
level & $\|u_h- Q_0 u\|_0 $  &rate & $\3bar u_h- u\3bar $ &rate & dim  \\
\hline
 &\multicolumn{5}{l}{by $P_4$ elements with strongly enforced boundary condition} \\ \hline
 4&   0.6433E-06 & 4.96&   0.7511E-04 & 3.98&    1762\\
 5&   0.2021E-07 & 4.99&   0.4699E-05 & 4.00&    7362\\
 6&   0.6320E-09 & 5.00&   0.2934E-06 & 4.00&   30082\\
 \hline
 &\multicolumn{5}{l}{by $P_4$ elements with weakly enforced boundary condition} \\ \hline
 4&   0.6781E-06 & 5.03&   0.7116E-04 & 3.90&    1920\\
 5&   0.2076E-07 & 5.03&   0.4577E-05 & 3.96&    7680\\
 6&   0.6407E-09 & 5.02&   0.2896E-06 & 3.98&   30720\\
 \hline
&\multicolumn{5}{l}{by $P_5$ elements with strongly enforced boundary condition} \\ \hline
 4&   0.2306E-07 & 5.94&   0.3385E-05 & 5.01&    2498\\
 5&   0.3668E-09 & 5.97&   0.1050E-06 & 5.01&   10370\\
 6&   0.5825E-11 & 5.98&   0.3266E-08 & 5.01&   42242\\ \hline
 &\multicolumn{5}{l}{by $P_5$ elements with weakly enforced boundary condition} \\ \hline
 4&   0.2481E-07 & 6.04&   0.3223E-05 & 4.94&    2688\\
 5&   0.3811E-09 & 6.02&   0.1024E-06 & 4.98&   10752\\
 6&   0.5938E-11 & 6.00&   0.3225E-08 & 4.99&   43008\\
 \hline
\end{tabular}%
\end{table}%

In the next computation,  we use a family of polygonal grids (with 12-side polygons)
   shown in Figure \ref{12gon}.
We let the polynomial degree $j=k+2$ for the weak gradient on such polygonal meshes.
The rate of convergence is listed in Tables \ref{t3}-\ref{t4}.
The convergence history confirms the theory.

\begin{figure}[htb]\begin{center}\setlength\unitlength{1.5in}
    \begin{picture}(3.2,1.4)
 \put(0,0){\includegraphics[width=1.5in]{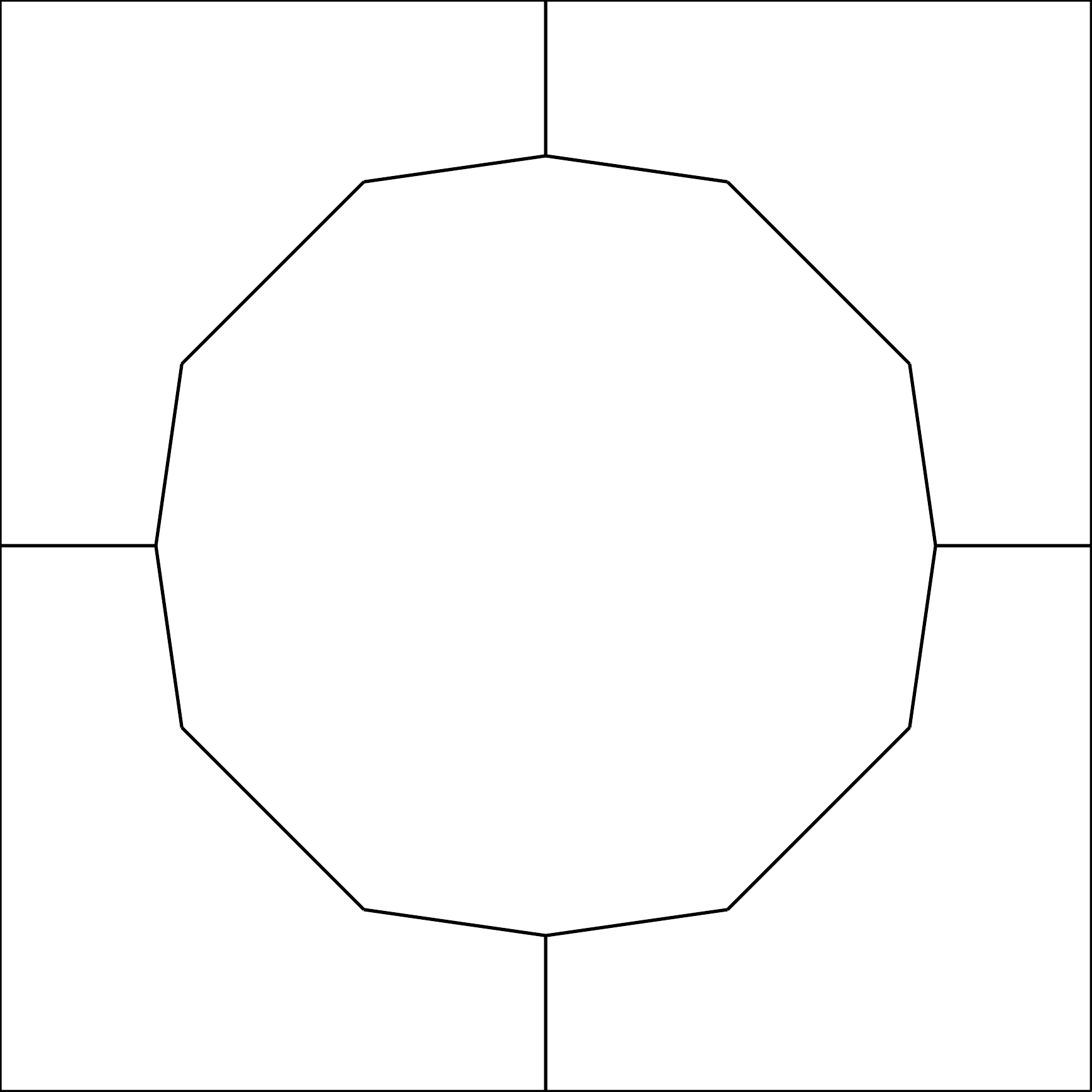}}  
 \put(1.1,0){\includegraphics[width=1.5in]{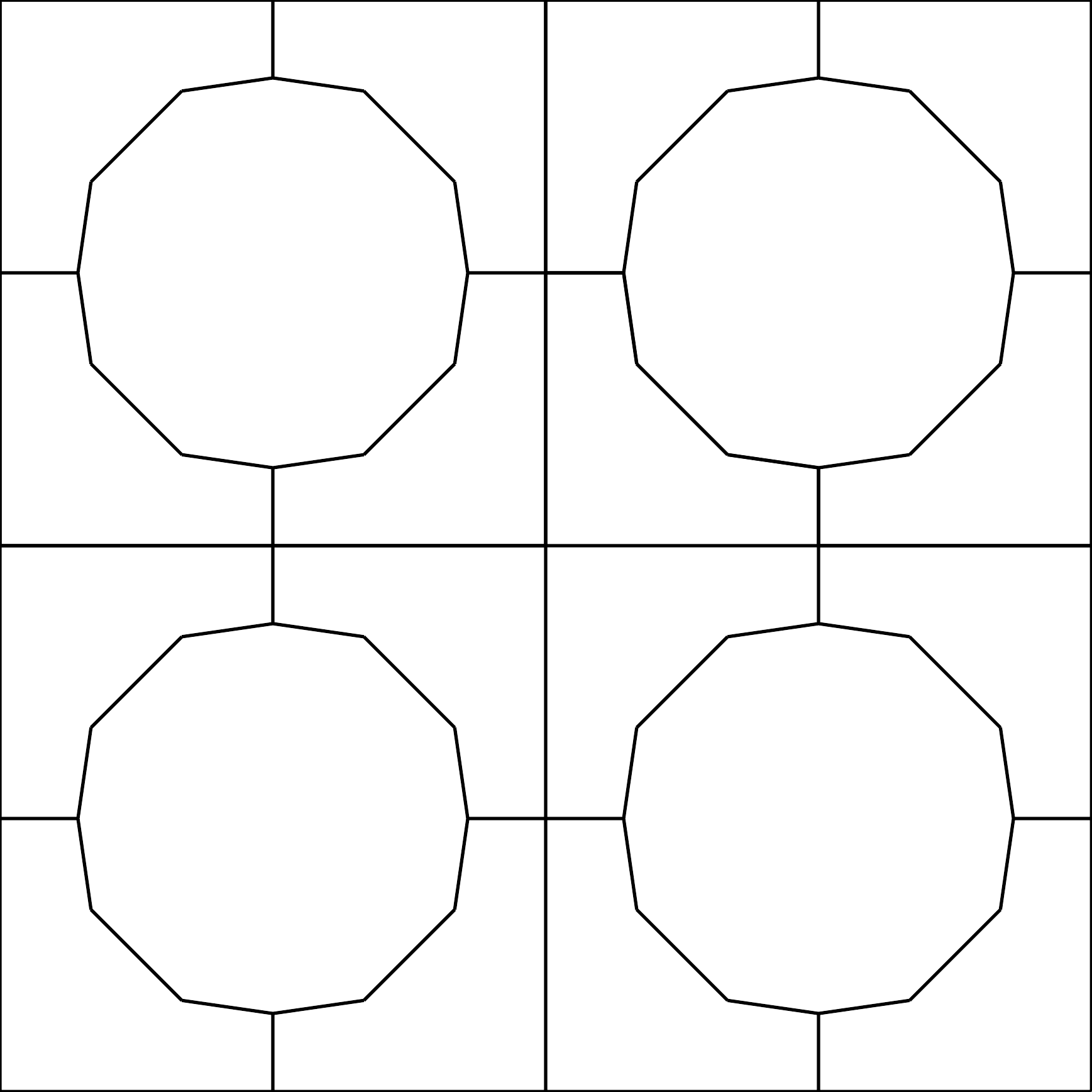}}
 \put(2.2,0){\includegraphics[width=1.5in]{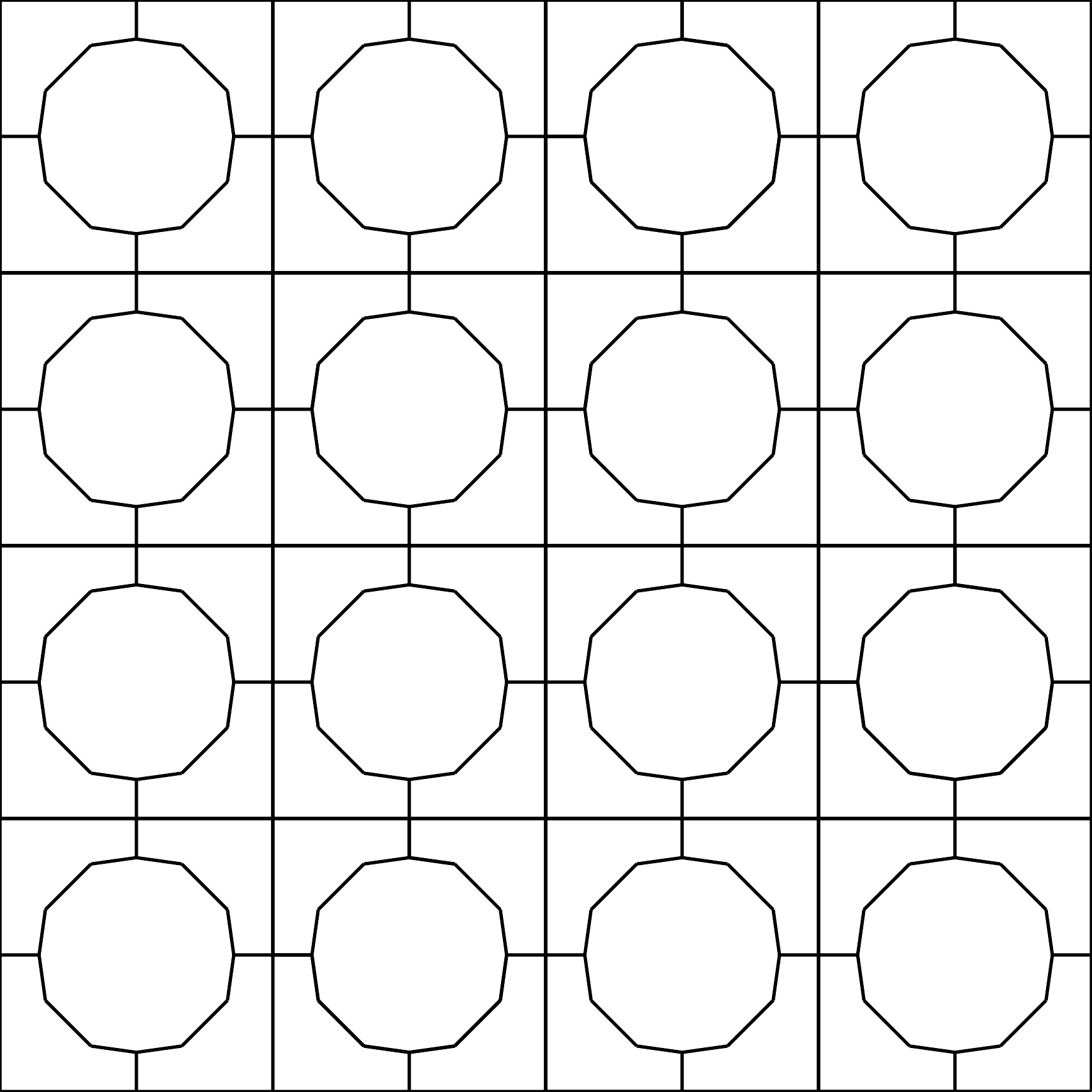}}
    \end{picture}
\caption{ The first three polygonal grids for the computation of Table \ref{t3}.  } \label{12gon}
\end{center}
\end{figure}

\begin{table}[h!]
  \centering \renewcommand{\arraystretch}{1.1}
  \caption{Error profiles and convergence rates for \eqref{s1} on polygonal grids
    shown in Figure \ref{12gon} }\label{t3}
\begin{tabular}{c|cc|cc|r}
\hline
level & $\|u_h- Q_0u\| $  &rate & $\3bar u_h-u\3bar $ &rate &dim \\
  \hline
 &\multicolumn{5}{l}{by $P_1$ elements with strongly enforced boundary condition} \\ \hline
 6&   0.2913E-03 & 2.00&   0.5402E-01 & 1.00&  15100\\
 7&   0.7289E-04 & 2.00&   0.2701E-01 & 1.00&  60924\\
 8&   0.1823E-04 & 2.00&   0.1351E-01 & 1.00&  244732\\
  \hline
 &\multicolumn{5}{l}{by $P_1$ elements with weakly enforced boundary condition} \\ \hline
 6&   0.2982E-03 & 2.03&   0.5333E-01 & 0.98&    15360\\
 7&   0.7374E-04 & 2.02&   0.2684E-01 & 0.99&    61440\\
 8&   0.1833E-04 & 2.01&   0.1346E-01 & 1.00&   245760\\
  \hline
 &\multicolumn{5}{l}{by $P_2$ elements with strongly enforced boundary condition} \\ \hline
 6&   0.1055E-05 & 3.00&   0.7604E-03 & 2.00&     30204\\
 7&   0.1318E-06 & 3.00&   0.1901E-03 & 2.00&   121852\\
 8&   0.1648E-07 & 3.00&   0.4753E-04 & 2.00&   489468\\
  \hline
 &\multicolumn{5}{l}{by $P_2$ elements with weakly enforced boundary condition} \\ \hline
 6&   0.1057E-05 & 3.01&   0.7574E-03 & 1.99&    30720\\
 7&   0.1320E-06 & 3.00&   0.1897E-03 & 2.00&  122880\\
 8&   0.1649E-07 & 3.00&   0.4748E-04 & 2.00&   491520\\
  \hline
 &\multicolumn{5}{l}{by $P_3$ elements with strongly enforced boundary condition} \\ \hline
 4&   0.2706E-05 & 3.99&   0.5478E-03 & 2.99&    3004\\
 5&   0.1696E-06 & 4.00&   0.6862E-04 & 3.00&   12412\\
 6&   0.1060E-07 & 4.00&   0.8582E-05 & 3.00&    50428\\
  \hline
 &\multicolumn{5}{l}{by $P_3$ elements with weakly enforced boundary condition} \\ \hline
 4&   0.2813E-05 & 4.04&   0.5421E-03 & 2.97&     3200\\
 5&   0.1728E-06 & 4.02&   0.6827E-04 & 2.99&   12800\\
 6&   0.1070E-07 & 4.01&   0.8561E-05 & 3.00&    51200\\
 \hline
    \end{tabular}%
\end{table}%

\begin{table}[h!]
  \centering \renewcommand{\arraystretch}{1.1}
  \caption{Error profiles and convergence rates for \eqref{s1} on polygonal grids
    shown in Figure \ref{12gon} }\label{t4}
\begin{tabular}{c|cc|cc|r}
\hline
level & $\|u_h- Q_0u\| $  &rate & $\3bar u_h-u\3bar $ &rate &dim \\
  \hline
 &\multicolumn{5}{l}{by $P_4$ elements with strongly enforced boundary condition} \\ \hline
 2&   0.7295E-04 & 3.68&   0.4484E-02 & 2.85&     232\\
 3&   0.2322E-05 & 4.97&   0.2830E-03 & 3.99&    1068\\
 4&   0.7291E-07 & 4.99&   0.1773E-04 & 4.00&    4540\\
  \hline
 &\multicolumn{5}{l}{by $P_4$ elements with weakly enforced boundary condition} \\ \hline
 2&   0.7529E-04 & 3.74&   0.4413E-02 & 2.83&     300\\
 3&   0.2358E-05 & 5.00&   0.2806E-03 & 3.97&   1200\\
 4&   0.7348E-07 & 5.00&   0.1765E-04 & 3.99&   4800\\
  \hline
 &\multicolumn{5}{l}{by $P_5$ elements with strongly enforced boundary condition} \\ \hline
 2&   0.7161E-05 & 6.38&   0.5901E-03 & 5.31&     336\\
 3&   0.1141E-06 & 5.97&   0.1863E-04 & 4.99&   1516\\
 4&   0.1807E-08 & 5.98&   0.5836E-06 & 5.00&  6396\\
  \hline
 &\multicolumn{5}{l}{by $P_5$ elements with weakly enforced boundary condition} \\ \hline
 2&   0.7233E-05 & 6.42&   0.5875E-03 & 5.31&    420\\
 3&   0.1144E-06 & 5.98&   0.1859E-04 & 4.98&  1680\\
 4&   0.1808E-08 & 5.98&   0.5831E-06 & 5.00&  6720\\
 \hline
    \end{tabular}%
\end{table}%

\section*{Acknowledgment}
We would like to express our appreciation to Junping Wang for his valuable advice.


\begin{thebibliography}{99}

\bibitem{abcm}
D. Arnold, F. Brezzi, B. Cockburn and D. Marini, Unified analysis of
discontinuous Galerkin methods for elliptic problems,  SIAM J.
Numer. Anal., 39 (2002), 1749-1779.

\bibitem{Babu73} I. Babu\v{s}ka, The finite element method with penalty,  Math. Comp., 27 (1973), 221-228.

\bibitem{bos}
S. Brenner, L. Owens and L. Sung, A weakly over-penalized symmetric interior penalty method,
Ele. Trans. Numer. Anal., 30 (2008), 107-127.


\bibitem{cgl}
B. Cockburn, J. Gopalakrishnan, and R. Lazarov, Unified hybridization of discontinuous
Galerkin, mixed, and conforming Galerkin methods for second order elliptic problems, SIAM J. Numer. Anal., 47 (2009), 1319-136.


\bibitem{cs1998}
B. Cockburn and C. Shu, The local discontinuous Galerkin finite element method for
convection-diffusion systems, SIAM J. Numer. Anal., 35 (1998), 2440-2463.


\bibitem{DoDu76}
J. Douglas Jr. and T. Dupont, Interior penalty procedures for elliptic and parabolic Galerkin methods, Computing Methods in Applied Sciences, (1976), 207-216.

\bibitem{Lipnikov2011}
K. Lipnikov, G. Manzini, F. Brezzi and A. Buffa,
The mimetic finite difference method for the 3D magnetostatic
field problems on polyhedral meshes, J. Comput. Phys., 230 (2011), 305-328.

\bibitem{liu}
J. Liu, S. Tavener, Z. Wang,
Lowest-order weak Galerkin finite element method for Darcy flow on convex polygonal meshes, SIAM J. Sci. Comput., 40 (2018), 1229-1252.

\bibitem{mwy-se}
L. Mu, J. Wang, and X. Ye, weak Galerkin finite element method for second-order elliptic problems on polytopal meshes, International Journal of Numerical Analysis and Modeling, 12 (2015), 31-53.

\bibitem{mwg1}
 L. Mu, X. Wang and X. Ye, A modified weak Galerkin finite element method for the Stokes equations, J. Comput. Appl. Math., 275 (2015), 79-90.

\bibitem{mwwy}
L. Mu, J. Wang, Y. Wang and X. Ye, A weak Galerkin mixed finite element method for biharmonic equations, Numerical Solution of Partial Differential Equations: Theory, Algorithms, and Their Applications, 45 (2013), 247-277.

\bibitem{mwy-biharm}	
L. Mu, J. Wang, and X. Ye, A weak Galerkin finite element method for biharmonic equations on polytopal meshes, Numerical Methods for Partial Differential Equations, 30 (2014), 1003-1029.

\bibitem{de}
D. Pietro and  A. Ern, Hybrid high-order methods for
variable-diffusion problems on general meshes, Comptes Rendus
Mathmatique,  353 (2015), 31-34.

\bibitem{ReHi73}
W. Reed and T. Hill. Triangular mesh methods for the neutron transport equation. Technical Report LA-UR-73-0479, Los Alamos Scientific Laboratory, Los Alamos, NM, 1973.

\bibitem{wy}
J. Wang and X. Ye, A weak Galerkin finite element method for second-order elliptic problems. J. Comput. Appl. Math. 241 (2013), 103-115.

\bibitem{wymix} J. Wang and X. Ye, A Weak Galerkin mixed finite element method for second-order elliptic problems, Math. Comp., 83 (2014), 2101-2126.

\bibitem{mwg}
X. Wang, N. Malluwawadu, F Gao and T. McMillan, A modified weak Galerkin finite element method, J. Comput. Appl. Math., 217 (2014), 319-327.

\bibitem{Whee78}
M. Wheeler, An elliptic collocation-finite element method with interior penalties. SIAM J. Numer. Anal., 15 (1978), 152-161.

\bibitem{cdg1}
X. Ye and S. Zhang, A conforming discontinuous Galerkin finite element method, International Journal of Numerical Analysis and Modeling, accepted.



\end{thebibliography}
\end{document}